\documentclass[fleqn]{article}
\usepackage{amsmath, amssymb, graphics, url, multirow, subfig}

\usepackage{amsthm}

\newcommand{\mathsym}[1]{{}}
\newcommand{\unicode}[1]{{}}



\begin{document}

\newtheorem*{theorem}{Theorem}

\title{Repeatedly Appending Any Digit to Generate Composite Numbers}

\author{Jon Grantham, Witold Jarnicki, John Rickert, and Stan Wagon}

\date{}

\maketitle

\begin{abstract}
We investigate the problem of finding integers $k$ such that appending any number of copies of the base-ten digit
$d$ to $k$ yields a composite number. In particular, we prove that there exist infinitely many integers coprime to all digits such that repeatedly appending {\it any} digit yields a composite number.
\end{abstract}

\section{Introduction.}

Recently L. Jones \cite{jones} asked about integers that yield only composites when a sequence of the same base-ten digit is appended to the right. He showed that $37$ is the smallest number with this property when appending the digit $d=1$.
For each digit $d\in\{3, 7, 9\}$, he also found numbers coprime to $d$ that yield only composites upon appending $d$s. 


In this paper we find a single integer that works for all digits simultaneously. More precisely,

\begin{theorem}
There are infinitely many positive integers $k$ with
$\gcd(k, 2\cdot 3 \cdot 5 \cdot 7) = 1$, such that for any base-ten digit $d$, appending any number of $d$s to $k$ yields a
composite number.
\end{theorem}

Further, we investigate the question of the smallest numbers that remain composite upon appending strings of a digit for each particular digit. Jones found, for digits $3$, $7$, $9$, respectively, the examples $4070$, $606474$, and $1879711$. It appears that $4070$ is the smallest for $d=3$; for digit $7$ we found $891$, which is almost certainly minimal; and for digit $9$, the likely answer $10175$ was discovered by \cite{rodenkirch}. In the next section, we explain the obstructions to proving that these three answers are the smallest.

\section{Seeds.}

Given a digit $d$, let's use the term {\it seed} for a number coprime to $d$ such that appending any number of $d$s on the right yields a composite. The smallest positive integer with this property will be referred to as a {\it minimal seed}. Only the cases $d\in
\{1,3,7,9\}$ are nontrivial. Jones proved that $37$ is the minimal seed for $d=1$, and he also found the seed $4070$ for digit $3$. For every $k<4070$, except $817$, we have found a value of $n$ such that appending $n$ $3$s yields a prime or, in three cases, a probable prime. For $817$, appending up to $554789$ $3$s yielded only composites. But factorizations show no apparent obstruction to primality, so we conjecture that $4070$ is the minimal seed for digit $3$.

A key concept in this area is the notion of a covering set, a concept introduced by P. Erd\H{o}s \cite{erdos}. Such a set corresponds to a finite list of primes such that every member of a given sequence is divisible
by one of the primes. Here the sequences are the numbers, which we call $s_n$, obtained by
appending $n$ copies of a digit $d$ to an initial value $k$; typically the numbers are proved composite
by finding a covering set. For example, when $n$ $7$s are appended to $891$, the resulting number
is divisible by $11$, $37$, $11$, $3$, $11$, or $13$ according to the mod-$6$ residue of $n$ (starting at $0$).

To see this, observe that $s_n$ is given by the formula
$$s_n = k \cdot {10}^n + \frac{d(10^n - 1)}9.$$
Because $10^6 \equiv 1$ modulo each of the four primes, easy modular arithmetic shows that
$s_{6m+i} \equiv 0 \pmod{p}$ for the cases $p = 11$, $13$, and $37$, where $i$, depending on $p$, is $0$, $2$, $4$, $5$, or $1$.
The same is true for $i = 3$, the case where $p = 3$, because $10^{6m+3} - 1$ is divisible by $27$, thus
eliminating the denominator of $9$ in these cases. This proves that $891$ is a seed for digit $7$.

When a sequence of primes $(p_0, p_1,\dots , p_{r-1})$ divides the corresponding sequence of terms $s_n$ for a
digit $d$ and seed $k$, we say that the primes form a {\it prime cover} for $(k,d)$. For example, $(11,37,11,3,11,13)$ is a prime cover for $(891,7)$.

We have shown that $891$ is a minimal seed for digit $7$, under the assumption that appending $11330$ $7$s to $480$, and $28895$ $7$s to $851$ yields primes. Each of these two large numbers has passed $200$ strong pseudoprime tests. For all other potential seeds below $891$, we have found primes that can be certified using elliptic curve methods with \textit{Mathematica} or \textit{Primo} \cite{primo}. We used $\text{\textit{Primo}}$ on the largest cases; the largest was $9777\ldots 7$ with $2904$ $7$s, which took $45$ hours.

The digit-$9$ case asks for an integer $k$ such that $(k+1)10^n-1$ is always composite; it is thus a variation on the classic Riesel problem \cite{keller, riesel, noprime, primegrid}, which addresses the same question in base $2$. For that classic case, it is known that $509202$ is a seed, meaning that $509203\cdot 2^n-1$ is composite for $n\ge 0$. Participants in the Riesel project have also investigated the decimal case, and showed \cite{rodenkirch}  that the expected minimal seed for digit $9$ is $10175$. To see that this is a seed, we again consider the number of appended digits modulo $6$ and find a prime cover: in this case $(11, 7, 11, 37, 11, 13)$.  Of the numbers smaller than $10175$, only two, $4420$ and $7018$, have not been eliminated as seeds. The Riesel project \cite{noprime, primegrid} has checked each through the addition of $750000$ $9$s without finding a prime. In this case, primality proving for a probable prime is easy using the Lucas $n+1$ test \cite{cp}.

Coverings are not the only tool in these investigations, since sometimes factorizations
yield all the compositeness that is sought. Consider the situation with digit $1$ but working in
base $b=m^2$ with $m$ odd. The minimal seed in all such cases is $1$ because, for $n$ appended $1$s to
the seed $1$, with $n$ even, the factorization
$$111\dots11_b=\frac{b^{n+1}-1}{b-1}=\left(\frac{m^{n+1}-1}{m-1}\right)\left(\frac{m^{n+1}+1}{m+1}\right)$$
yields integer factors, and so the result is composite. When $n$ is odd, the total number of $1$s is
even, so compositeness is clear. Similar factorization methods show that the minimal seed for
digit $1$ in base $4$ is $5$, for digit $3$ in base $4$ is $8$, and for digit $8$ in base $9$ is $3$.

\section{A pandigital seed.}

It is not hard to find an integer that remains composite when any sequence of the form $ddd \ldots d$ is appended on the right, where 
$d$ is any decimal digit. We leave it as an exercise to show that $6930$ does the job; only the case $d=1$ requires a prime cover, and the one used in $\S 2$ for $891$ --- $(11, 37, 11, 3, 11, 13)$ --- works. Some prime searching shows that $6930$ is the smallest such example (the most difficult candidate to eliminate was $6069$; $1525$ $1$s yielded a prime).

A more natural problem in our context is to consider only the digits $1$, $3$, $7$, $9$ and ask for an integer $k$ that is a seed for each of these four digits (thus $k$ is coprime to 3 and 7). We call such a positive integer $k$ a {\it pandigital seed}.

For a prime $p$ coprime to $10$, we use the term {\it period} of $p$ to mean the smallest positive integer $r$ so that, for all $n$, $s_{n+r}\equiv s_n \pmod{p}$. The period of $3$ is $3$, while for other primes it is simply the order of $10$ modulo $p$.
If the period of a prime $p$ is small, then $p$ may divide a large proportion of the terms of the sequence $s_n$. In particular, if the period is $r$, then either every $r$th term of $\{s_n\}$ is divisible by $p$ or no terms of the sequence are divisible by $p$.

\begin{theorem}
A pandigital seed exists. An example is $4942768284976776320$.
\end{theorem}

\begin{proof}
A proof requires only checking that particular covers work, but we outline
the method by which the large seed and corresponding prime covers were
found. We find, for each digit, a prime cover so that the congruence
conditions on $k$ arising from the four covers do not contradict each other.
This method of coherent prime covers was used in \cite{bff,ffk,lm} to find
infinitely many values $k$ so that both $k2^n+1$ and $k2^n-1$ are composite
for all $n$, and solve related problems. To find such covers, we first need to
analyze the condition that a term in the sequence $\{s_n\}$ is divisible by
a given prime $p$.

If we assume that $p \notin \{2,3,5\}$, then $s_n\equiv  0\pmod{p}$ if and only if
$p$ divides  $$9k\cdot 10^n + d(10^n-1),$$
which is equivalent to 
\begin{equation}
k\equiv 9^{-1} d(10^{-n}-1)\pmod{p}.
\end{equation}
If $p=3$, then we instead have the condition 
$$s_n\equiv k+ d\frac{10^n-1}{9}\equiv 0\pmod{3},$$
which, because $(10^n-1)/9\equiv n \pmod{3}$, reduces to $k\equiv 2dn\pmod{3}$.
It is useful to observe that when $n$ is even then $10^n\equiv 1\pmod{11}$,
so that in this case $s_n$ is congruent modulo 11 to the seed itself. Therefore, the condition 
$k\equiv 0\pmod{11}$ makes $11$ a factor of $s_n$ whenever $n$ is even. 
Hence we may focus on forcing composites for odd values of $n$.

Since the period of $p=37$ is $3$, we consider this prime next. 
When the number of appended digits is $n=6i+3$, equation \thetag{1} gives
$$k\equiv 10^{-(6i+3)}-1=\left(10^{-6}\right)^i\cdot 10^{-3}-1\equiv 0\pmod{37}.$$
Application of \thetag{1} to other values of $n$ shows that $37$ divides $s_n$ for $n\equiv 0,1,2\pmod{3}$ provided $k\equiv 0, 11d, 10d\pmod{37}$, respectively. If $k\equiv 0\pmod{37}$, then $37$ may be used as a prime divisor no matter which digit is appended. Therefore we can assume $k\equiv 0\pmod{37}$, and so we have that $s_n$ is divisible by $11$ when $n\equiv 0$, $2$, or $4\pmod{6}$ and by $37$ when $n\equiv 0$ or $3\pmod{6}$. This leaves only the eight cases $n\equiv 1$ or $5\pmod{6}$ with digits $1$, $3$, $7$, and $9$ to be taken care of by other primes, as shown in Table \ref{tab:1137}.

\begin{center}
\begin{table}[h]
\centering
\begin{tabular}{|c||c|c|c|c|c|c|}
\hline
 & \multicolumn{6}{|c|}   {$n\pmod{6}$} \\
\hline & & & & & & \\[-1em] \hline
digit & $0$ & $1$ & $2$ & $3$ & $4$ & $5$ \\ \hline
$1$ & $11$ & $?$ & 11 & 37 & $11$ & $?$ \\ \hline
3 & 11 & ? & 11 & 37 & 11 & $?$ \\ \hline
7 & 11 & ? & 11 & 37 & 11 & $?$ \\ \hline
9 & 11 & ? & 11 & 37 & 11 & $?$  \\ \hline
\end{tabular}
\caption{Divisors of $s_n$ for digit $d$ using primes 11 and 37 with a seed that satisfies $k\equiv 0\pmod{11\cdot 37}$.}
\label{tab:1137}
\end{table}
\end{center}

To find divisors of $s_n$ for $n\equiv 1$ or $5\pmod 6$, we note that the primes $7$ and $13$ have period $6$.
Solving congruence \thetag{1} leads to the conditions listed in Table \ref{tab:713}.  These show that if $k\equiv 2\pmod{7}$, then two of the eight cases are divisible by $7$: the digit $1$ with $n\equiv 1\pmod{6}$ and digit $9$ with $n\equiv 5\pmod{6}$ cases. Similarly, any of $k\equiv 1$, $3$, or $9\pmod{13}$ provides divisibility for two of the cases. Each of these cases is then combined with a set of additional primes that contains $3$, $101$, $41$, $271$, $73$, and $137$, all of which have period $8$ or less. Finally, a computer search found a list of primes that handles all cases.

\begin{table}[h]
\begin{tabular}{|c||c|c|c|c|}
\hline
  & $n \equiv 1 \pmod{6}$ & $n \equiv 5 \pmod{6}$ &  $n \equiv 1 \pmod{6}$ & $n \equiv 5 \pmod{6}$ \\ \hline & & & & \\[-1em] \hline
digit 1 & $k \equiv 2 \pmod{7}$ & $k \equiv 1 \pmod{7}$ & $k \equiv 9 \pmod{13}$ & $k \equiv 1 \pmod{13}$ \\ \hline
digit 3 & $k \equiv 6 \pmod{7}$ & $k \equiv 3 \pmod{7}$ & $k \equiv 1 \pmod{13}$ & $k \equiv 3 \pmod{13}$ \\ \hline
digit 7 & $k \equiv 0 \pmod{7}$ & $k \equiv 0 \pmod{7}$ & $k \equiv 11 \pmod{13}$ & $k \equiv 7 \pmod{13}$ \\ \hline
digit 9 & $k \equiv 4 \pmod{7}$ & $k \equiv 2 \pmod{7}$ & $k \equiv 3 \pmod{13}$ & $k \equiv 9 \pmod{13}$ \\ \hline
\end{tabular}
\caption{Conditions on $k$ to guarantee that 7 or 13 divides the number obtained by appending a digit string to $k$.}
\label{tab:713}
\end{table}

The smallest value of $k$ found so far uses the primes 
$3$, $7$, $11$, $13$, $31$, $37$, $41$, $73$, $101$, $137$, $211$, $241$, and $271$. The cover-lengths for the four digit-cases are $6$, $6$, $30$, and $8$, respectively.
The prime covers for the four digits are as follows,

\begin{eqnarray*}
d=1 :&  (11,3,11,37,11,13),        \\
d=3 :&  (11,13,11,37,11,7), \\
d=7 :& (11,3,11,37,11,271,11,3,11,37,11,41,11,3,11,37,11,31,\\
 &\  11,3,11,37,11,211,11,3,11,37,11,241),\\
d=9 :&  (11,73,11,101,11,137,11,101).\\
\end{eqnarray*}

Tables 3 and 4 show the correspondence between the values of $n$ and $k$ for each digit. For example, when we are appending $n$ $7$s to $k$ where $n\equiv 11 \pmod{30}$, we see that $41$ divides $s_n$ whenever $k\equiv 28\pmod{41}$.

We apply the Chinese Remainder Theorem to all of the conditions on $k$ in Tables 3 and 4 to find the pandigital seed $k=4942768284976776320$.\end{proof}

Because $k$ is not divisible by $3$ or $7$, we can add $k\equiv 1 \pmod {10}$ to the conditions used in the proof, which then gives us $$1970728582053685108721 \pmod{19657858137687083324010}\text{,}$$ a value of $k$ that satisfies the theorem as stated in $\S 1$. This yields infinitely many such values.

 
\begin{table} 
\subfloat{\begin{tabular}{|l|l|}
  \multicolumn{2}{c}   {digit 1} \\ \hline
 \text{classes for }\text{\textit{$n$}} & classes for $k$\\
  \hline & \\[-1em] \hline
  $0 \pmod {2}$ & $0 \pmod {11}$ \\ \hline
  $1 \pmod {6}$ & $2 \pmod {3}$ \\ \hline 
  $3 \pmod {6}$ & $0 \pmod {37}$ \\ \hline
  $5 \pmod {6}$ & $1 \pmod {13}$ \\
\hline
\end{tabular}}%
\qquad\qquad%
\subfloat{\begin{tabular}{|l|l|}
  \multicolumn{2}{c}   {digit 3} \\ \hline
 \text{classes for }\text{\textit{$n$}} & classes for $k$\\
  \hline & \\[-1em] \hline
  \text{0 (mod 2)} & 0 (mod 11) \\ \hline
   \text{1 (mod 6)} & 1 (mod 13) \\ \hline 
    \text{3 (mod 6)} & 0 (mod 37) \\ \hline
 \text{5 (mod 6)} & 3 (mod 7) \\
 \hline
\end{tabular}}
\caption{Residue classes for the seed $k$ that guarantee the compositeness of $s_n$ when $1$ or $3$ is appended.}

\end{table}
 

\begin{table}[hbpt]
\subfloat{\begin{tabular}{|r@{}l|r@{}l|}
  \multicolumn{4}{c}   {digit 7} \\ \hline
 \multicolumn{2}{|c|}{classes for $n$} & \multicolumn{2}{|c|}{classes for $k$}\\
  \hline & \\[-1em] \hline
  $0$&$\pmod {2}$ & $0$&$\pmod {11}$ \\ \hline
  $1$&$\pmod {6}$ & $2$&$\pmod {3}$ \\ \hline 
  $3$&$\pmod {6}$ & $0$&$\pmod {37}$ \\ \hline
  $5$&$\pmod {30}$ & $0$&$\pmod {271}$ \\ \hline
  $11$&$\pmod {30}$ & $28$&$\pmod {41}$ \\ \hline
  $17$&$\pmod {30}$ & $20$&$\pmod {31}$ \\ \hline
  $23$&$\pmod {30}$ & $106$&$\pmod {211}$ \\ \hline
  $29$&$\pmod {30}$ & $7$&$\pmod {241}$ \\ \hline
\end{tabular}}%
\qquad\qquad%
\subfloat{\begin{tabular}{|r@{}l|r@{}l|}
  \multicolumn{4}{c}   {digit 9} \\ \hline
 \multicolumn{2}{|c|}{classes for $n$} & \multicolumn{2}{|c|}{classes for $k$}\\
  \hline & \\[-1em] \hline
  $0$&$\pmod {2}$ & $0$&$\pmod {11}$ \\ \hline
  $1$&$\pmod {8}$ & $21$&$\pmod {73}$ \\ \hline 
  $3$&$\pmod {4}$ & $9$&$\pmod {101}$ \\ \hline
  $5$&$\pmod {8}$ & $40$&$\pmod {137}$ \\ \hline
\end{tabular}}%

\caption{Residue classes for the seed $k$ that guarantee the compositeness of $s_n$ when $7$ or $9$ is appended.}
\end{table}

\section{Open problems.}

We conclude with some unsolved problems.

\begin{enumerate}
\item Find a number of $3$s which can be appended to $817$ to obtain a probable prime, thus completing the proof, modulo probable primes, that $4070$ is minimal for the digit $3$.
\item Find a number of $9$s which can be appended to $4420$ or $7018$ to produce a prime.
\item Certify primality of $480$ with $11330$ $7$s appended and $851$ with $28895$ $7$s. Doing so would complete the digit-$7$ case.
\item Data for all bases up to $10$ can be found at \cite{rickert}. Similar problems exist for these bases. 
\item Find a base-ten pandigital seed that is smaller than $4942768284976776320$.
\item Investigate for various bases the situation where the appended digits come from a fixed sequence, as was done by Jones and White \cite{joneswhite} for base ten.
\end{enumerate}

The authors thank the referee for many valuable comments that helped improve the exposition.

\bibliography{gjrw}
\bibliographystyle{monthly}

\noindent\textit{Institute for Defense Analyses, Center for Computing Sciences, 17100 Science Drive, Bowie, MD\\grantham@super.org}

\bigskip

\noindent\textit{Google R \& E Krak{\' o}w, Rynek Glowny 12, 31-042 Krak{\' o}w, Poland\\witoldjarnicki@google.com}

\bigskip

\noindent\textit{Rose-Hulman Institute of Technology, 5500 Wabash Avenue, Terre Haute, IN 47803\\rickert@rose-hulman.edu}

\bigskip

\noindent\textit{Mathematics Department, Macalester College, St. Paul, MN 55105\\wagon@macalester.edu}

\end{document}